\documentclass[a4paper,11pt]{article} 
\usepackage{amsmath}
\usepackage[retainorgcmds]{IEEEtrantools}
\usepackage{graphicx}
\usepackage[margin=1.1in]{geometry}
\usepackage{amsfonts}
\usepackage{amsthm}
\usepackage{setspace}
\usepackage{mathtools}
\usepackage{listings}
\usepackage{mathabx}
\usepackage{amssymb}
\usepackage[mathscr]{euscript}
\usepackage{pictexwd,dcpic}
\newtheorem{thm}{Theorem}
\theoremstyle{definition}

\newtheorem{crl}{Corollary}
\newtheorem{prp}{Proposition}
\newtheorem{cnj}{Conjecture}

\newcommand{\Z}{\mathbb{Z}}

\newcommand{\R}{\mathbb{R}}
\newcommand{\RP}{\mathbb{R}\mathbb{P}}

\author{Cole Hugelmeyer} 
\title{Inscribed Trefoil Knots}
\begin{document} 
\maketitle 
\begin{abstract}
Let $K$ be a knot type for which the quadratic term of the Conway polynomial is nontrivial, and let $\gamma: \R\to \R^3$ be an analytic $\Z$-periodic function with non-vanishing derivative which parameterizes a knot of type $K$ in space. We prove that there exists a sequence of numbers $0\leq t_1 < t_2 < ... < t_6 < 1$ so that the polygonal path obtained by cyclically connecting the points $\gamma(t_1), \gamma(t_2), ..., \gamma(t_6)$ by line segments is a trefoil knot. 
\end{abstract}

\section{Introduction}
We investigate the existence of piecewise linear trefoils inscribed in a knot. In particular, we consider the problem of finding six points on a given knot which form a Trefoil knot when cyclically connected by straight line segments in the same order in which they appear. Our main result is as follows.

\begin{thm}
Let $K$ be a knot type for which the quadratic term of the Conway polynomial is nontrivial, and let $\gamma: \R\to \R^3$ be an analytic $\Z$-periodic function with non-vanishing derivative which parameterizes a knot of type $K$ in space. Then there exists a sequence of numbers $0\leq t_1 < t_2 < ... < t_6 < 1$ so that the polygonal path obtained by cyclically connecting the points $\gamma(t_1), \gamma(t_2), ..., \gamma(t_6)$ by line segments is a trefoil knot. 
\end{thm}

For the proof, we will begin by defining a submanifold of configuration space which lies inside the closure of the set of 6-tuples which form trefoils. Then, we will use some intersection theory to prove that if the quadratic term of the Conway polynomial of our knot is nontrivial then a 1-parameter family of 6-tuples of points on our knot lies in this submanifold. Finally, we will use some geometric arguments relying on the analyticity of the parameterization to show that we can perturb one of those 6-tuples to make a trefoil. 

It appears quite challenging to remove the requirement that the curve be analytic. The same sort of difficulties that arise in removing regularity requirements in the Toeplitz inscribed square problem also arise in the problem of inscribed trefoils. A simple limiting argument fails to work because the trefoils cannot be guaranteed not to degenerate to planar configurations, much like how when one approximates a Jordan curve by smoothings, the inscribed squares cannot be guaranteed not to shrink to zero. In order to relax the regularity requirements in these kind of problems, we likely need new geometric insights. \cite{survey}

\section{A Submanifold of Configuration Space}

If $X$ is a manifold, then $C_n(X)$ will denote the space of all n-tuples of distinct points of $X$. This is called the n-th (labeled) configuration space of $X$. In this section we will construct a submanifold of $C_6(S^3)$ with some interesting properties.

Consider a small geometrically spherical 4-ball in real projective 4-space, $B^4\subseteq \RP^4$.  Given a point $p$ not inside of $B^4$, we can consider the lines in $\RP^4$ that pass through $p$ as well as some part of $B^4$. These lines either intersect $\partial B^4$ at two distinct points, or they lie tangent to $\partial B^4$, intersecting it at exactly one point. For fixed $p$, let $T_p$ be the set of points in $\partial B^4$ at which some line passing through $p$ and tangent to $\partial B^4$ intersects $\partial B^4$. The set $T_p$ is a 2-sphere that divides $\partial B^4$ into two regions. Furthermore, for any line that passes through $p$ and intersects $\partial B^4$ at two distinct points, the two intersection points will lie on different sides of $T_p$. This means that if $\ell_1,\ell_2,\ell_3$ are three distinct lines that each pass through $p$ and each intersect $\partial B^4$ at two points, then there is a well-defined partition of the six intersection points into two groups of three which is determined by grouping together intersection points that lie on the same side of $T_p$. Let us define a graph with the six intersection points as its vertices and and edge between two vertices if they are on the same side of $T_p$ or they lie on the same line from the set $\{\ell_1,\ell_2,\ell_3\}$. We will call this graph $G(\ell_1,\ell_2,\ell_3)$. Regardless of which three lines we pick, this graph will be isomorphic to the edge graph of a triangular prism. One interesting property of this graph is that its compliment graph is cyclic. Therefore, if we define $\text{Cyc}_6$ to be the graph with vertices $\{1,...,6\}$ and edges between any two numbers that differ by 1 mod 6, then we may select a graph isomorphism from the compliment graph $G^c(\ell_1\ell_2\ell_3)$ to the fixed cyclic graph $ \text{Cyc}_6$. We now define $M$ to be the moduli space for the following data: 
\begin{itemize}
\item[1)] A point $p\in \RP^4\setminus B^4$.
\item[2)] An \emph{unordered} triple of distinct lines $\{\ell_1,\ell_2,\ell_3\}$ that each pass through $p$ and each intersect the 3-sphere $\partial B^4$ at exactly two points.
\item[3)] A specified graph isomorphism $G^c(\ell_1,\ell_2,\ell_3)\to \text{Cyc}_6$.
\end{itemize}

\begin{figure}[h]
\caption{A representative for a point in $M$. }
\centering
\includegraphics[scale = 0.75]{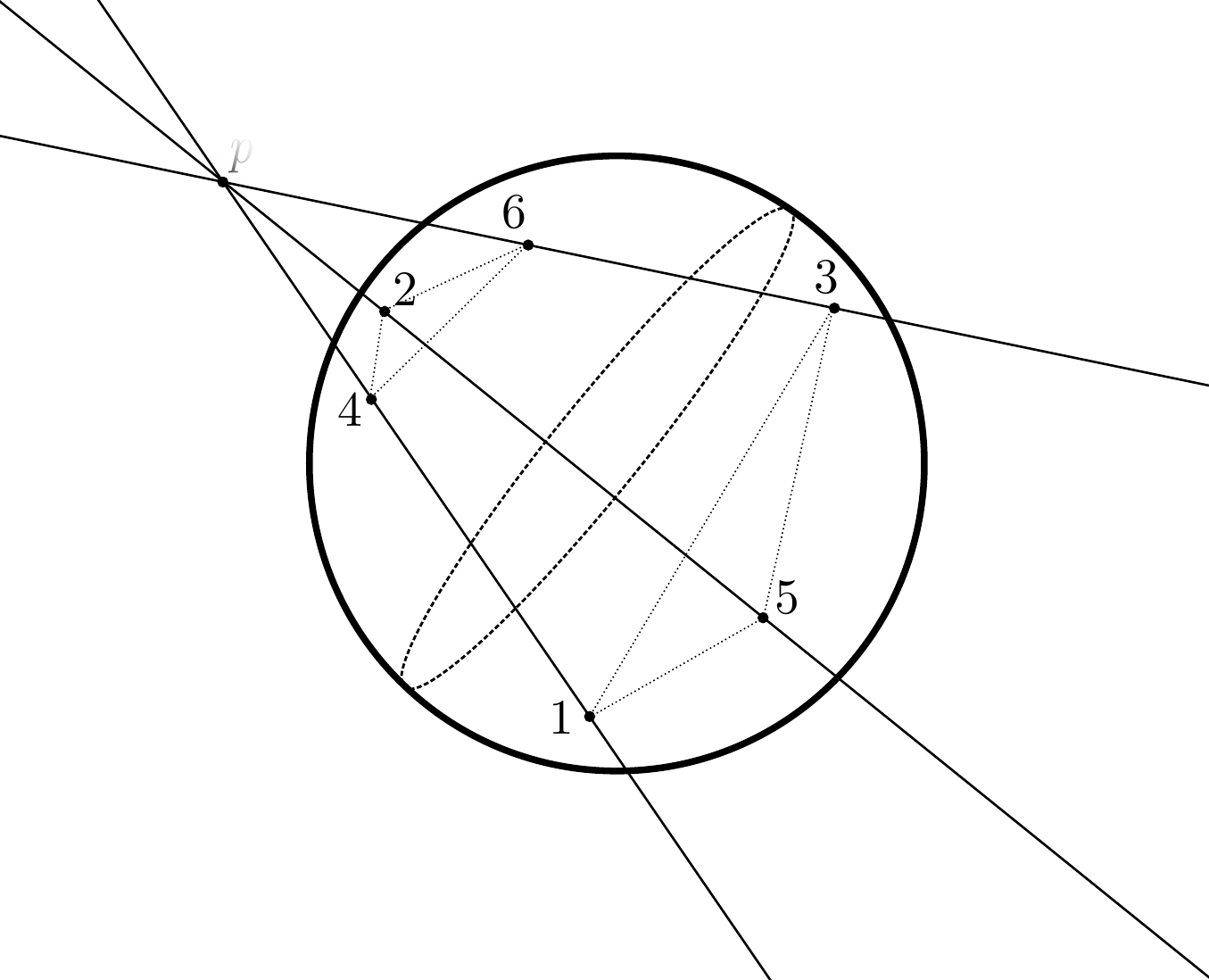}
\end{figure}

Alternatively, we could have simply defined $M$ as the subset of $C_6(S^3)$ consisting of $6$-tuples $(x_1,...,x_6)$ so that (when we identify $S^3$ with $\partial B^4\subseteq \RP^4$) the lines going through the pairs of points $(x_1,x_4)$, $(x_2,x_5)$, and $(x_3,x_6)$ all meet at a single point in $\RP^4$. This makes it clear that $M$ admits an embedding into $C_6(S^3)$ by taking a point in $M$ to the 6-tuple of intersection points ordered via the labeling induced by the map to $\text{Cyc}_6$. It should be noted that when stereographically projected to $\R^3$, an element of $M$ generically consists of the set of vertices for a triangular prism inscribed in a sphere.

\begin{prp}\label{diffeoM}
The manifold $M$ is diffeomorphic to $\R\times S^3 \times C_3(\R^3)$.
\end{prp}

\begin{proof}
The tangent bundle of $\RP^3$ is trivial, so there is a choice of diffeomorphism, continuously depending on $p\in \RP^4\setminus B^4$, from $\R^3$ to the space of lines that pass through $p$ and intersect $\partial B^4$ at exactly two points. Furthermore, the data of a specified graph isomorphism $G^c(\ell_1,\ell_2,\ell_3)\to \text{Cyc}_6$ is equivalent to a specified ordering of the three lines along with a specified side of the partition determined by $T_p$. To see why this equivalence holds, we see that the ordering of $\ell_1,\ell_2,\ell_3$ allows us to determine an orientation for the graph $G^c(\ell_1,\ell_2,\ell_3)$, and the specified side of the partition determined by $T_p$ allows us to select a single base point out of the six intersection points by choosing the intersection point of the first line which lies in the specified side of the partition. The base point along with the orientation defines an isomorphism. Now, we have enough information to say that $M$ is diffeomorphic to a double cover for $(\RP^4\setminus B^4 )\times C_3(\R^3)$. Therefore, to complete the proof we simply need to check that $\Z/2 \simeq \pi_1(\RP^4\setminus B^4)$ acts nontrivially on the orientation of $G^c(\ell_1,\ell_2,\ell_3)$ determined by the specified isomorphism to $\text{Cyc}_6$. This is straightforward to check, as this $\Z/2$ action corresponds to a $180$ degree rotation of our configuration of points in $S^3$, and this reverses the orientation of the cyclic order.
\end{proof}

\begin{crl}
The manifold $M$ is $13$ dimensional and orientable.
\end{crl}

\begin{proof}
$\R\times S^3 \times C_3(\R^3)$ is $13$ dimensional and orientable.
\end{proof}

\begin{prp}\label{properM}
The manifold $M$ is properly embedded in $C_6(S^3)$.
\end{prp}

\begin{proof}
We can define three functions $M\to \R_{\geq 0}$ via the diffeomorphism to $ \R\times S^3 \times C_3(\R^3)$ that was constructed in the proof of Proposition \ref{diffeoM}. The first function, which we will call $f_1$, will take a point to the absolute value of the $\R $ component of that point. The second function $f_2$ will take a point in $M$ to the reciprocal of the minimum distance between two of the three points in $\R^3$ from the $C_3(\R^3)$ component. The third function $f_3$ will take a point in $M$ to the maximum distance from zero of the three points in $\R^3$ from the $C_3(\R^3)$ component. These functions have the property that any subset of $M$ on which all of these functions are bounded has compact closure. Therefore, we see that if $q_1,q_2,...$ is a sequence of points in $M$ that eventually leaves every compact set, then this sequence has a subsequence $q_{n_1}, q_{n_2}, ...$ such that for some $i\in \{1,2,3\}$, the sequence $f_i(q_{n_1}), f_i(q_{n_2}), ...$ goes to $+\infty$. Therefore, to show that $M$ is properly embedded in $C_6(S^3)$, it suffices to prove that if $i\in \{1,2,3\}$ and $q_1,q_2,...$ is a sequence of points in $M$ with $f_i(q_1), f_i(q_2),...$ going to $+\infty$, then $q_1,q_2,...$ does not have a limit in $C_6(S^3)$. For $i = 1$, such a sequence will have $p$ approaching $\partial B^4$ which will cause three of the intersection points to become arbitrarily close to each other. For $i = 2$, the minimum distance between a pair of intersection points will go to zero because two lines will get arbitrarily close. For $i = 3$, the distance between the two intersection points from some line will go to zero because some line will get arbitrarily close to being tangent to $\partial B^4$. Therefore, we see that $M$ must be properly embedded.
\end{proof}

Perhaps the most interesting aspect of $M$ is how it interacts with knots. Let $\gamma: \R\to S^3$ be a $\Z$-periodic parameterization of a knot. We define $Q_\gamma \subseteq C_6(S^3)$ to be the set of all 6-tuples of the form $(\gamma(t_1), ... , \gamma(t_6))$ such that $t_1 < t_2 < ... < t_6 < t_1 + 1$. We will see that the oriented intersection class of $M$ with $Q_\gamma$ is well-defined in $H_1(Q_\gamma)$ despite the fact that these two manifolds are non-compact.

\begin{prp}\label{diffeoQ}
$Q_\gamma$ is diffeomorphic to $\R^5 \times S^1$.
\end{prp}

\begin{proof}
We can select a point of $Q_\gamma$ by first selecting $t_1$, and then selecting positive real numbers $s_1,...,s_6$ such that $s_1 + ... + s_6 = 1$ and then setting $t_n = t_1 + s_1 + ... + s_{n-1}$. The only redundancy is the choice of $t_1$ which only matters mod $\Z$.  This demonstrates that a point of $Q_\gamma$ is determined uniquely by a point in $S_1$ and a point in the interior of the 5-simplex, meaning that $Q_\gamma$ is diffeomorphic to $\R^5\times S^1$.
\end{proof}

\begin{crl}
The manifold $Q_\gamma$ is $6$ dimensional and orientable
\end{crl}

\begin{proof}
 $\R^5 \times S^1$ is $6$ dimensional and orientable.
\end{proof}

\begin{prp}\label{properQ}
The manifold $Q_\gamma$ is properly embedded in $C_6(S^3)$.
\end{prp}

\begin{proof}
If a sequence of points in $Q_\gamma$ eventually leaves every compact set, then using the notation from Proposition \ref{diffeoQ}, the minimum value of $s_1,...,s_6$ must go to zero. This means that the minimum distance between two points must go to zero in $C_6(S^3)$.
\end{proof}

Let $\gamma: \R\to S^3$ be a $\Z$-periodic smooth parameterization of a knot with nonvanishing derivative. If we think of $S^3$ as the unit sphere in $\R^4$, we can define the thickness $\tau(\gamma)$ to be the infimum radius of all of the circles in $\R^4$ that lie in $S^3$ and pass through at least three points on the knot. For a smooth parameterization with nonvanishing derivative, the thickness is always a positive real number, and it varies continuously with respect to the smooth topology.

We will now prove a couple of of facts about thickness.

\begin{prp}\label{sep}
If $y_1$ and $y_2$ are two points on a knot $\gamma$ with a distance strictly less than $2\tau(\gamma)$, then, letting $P$ denote the 2-sphere in $S^3$ in which $y_1$ and $y_2$ are antipodal points, the knot $\gamma$ intersects $P$ transversely, and intersects $P$ only at the points $y_1$ and $y_2$. Thus, of the two pieces of the knot on either side of the points $y_1$ and $y_2$, one piece lies entirely on one side of $P$ and the other piece lies entirely on the other side of $P$.
\end{prp}

\begin{proof}

To demonstrate that $y_1$ and $y_2$ are the only intersection points with $P$, we simply note that if a third point $y_3$ lied on $P$ then the circle passing through $(y_1,y_2,y_3)$ would have radius less than $\tau(\gamma)$ which contradicts our assumptions. Now, to see that the knot must intersect $P$ transversely, observe that if the tangent line of the knot at $y_i$ (for $i= 1$ or $2$) were to be tangent to $P$, then taking $y_3$ to be on the knot and very close to $y_i$, the circle through these points will have a radius less than $\tau(\gamma)$ which contradicts our assumptions. 

\end{proof}

\begin{prp}\label{adjacent}
Given a finite set $F$ of points on a knot $\gamma$, if some pair of the points are within distance $2\tau(\gamma)$ of each other, then the pair with minimal distance are adjacent with respect to the cyclic order of the points on the knot (in the sense that there is a path on the knot between them that does not contain any other point from the set). 
\end{prp}

\begin{proof}
From the previous proposition, we see that if $y_1, y_2 \in F$ have minimal distance and that distance is less than $2\tau(\gamma)$, then the sphere $P$ with $y_1$ and $y_2$ as its antipodes has a path in $\gamma$ passing through the inside of this sphere and going from $y_1$ to $y_2$. If any point from $F$ lied on this path, then that point would be closer to $y_1$ than $y_2$ and this would contradict minimality.
\end{proof}

\begin{prp}\label{bound}
Let $\gamma: \R\to S^3$ be a smooth parameterization of a knot with nonvanishing derivative. Then, considering $S^3$ as the unit sphere in $\R^4$, for any 6-tuple of points $(x_1,...,x_n)\in M\cap Q_\gamma\subseteq C_6(S^3)$ and any $i\neq j$ from $1$ to $6$, we have $|x_i - x_j| \geq 2\tau(\gamma)$. \end{prp}

\begin{proof}

We will derive a contradiction from assuming that there exists a point $(x_1,...,x_6)\in M\cap Q_\gamma$ such that the pair of indices $i,j$ with minimal $|x_i - x_j|$ have $|x_i - x_j| < 2\tau(\gamma)$. By Proposition \ref{adjacent} we have that $i$ and $j$ are adjacent with respect to the cyclic ordering of the knot.  Recalling the construction of $M$, we have that $x_i$ and $x_j$ are from different lines, and they lie on different sides of the partition $T_p$. This means that if $x_k$ is the other point on the same line as $x_i$ and $x_\ell$ is the other point on the same line as $x_j$, then the four points  $x_i,x_j,x_k,x_\ell$ lie on a circle in such a way that $x_i$ and $x_j$ are non-adjacent. However, the closest pair of points on a cyclic polygon must be adjacent. This means that the pair of points of minimal distance from of the set $\{x_i,x_j,x_k,x_\ell\}$ are not $x_i$ and $x_j$, which contradicts our assumption of minimality.\end{proof}

\begin{prp}\label{inv}
Despite $M$ and $Q_\gamma$ being non-compact, the oriented intersection homology class $[M\cap Q_\gamma] \in H_1(Q_\gamma)$ is well-defined and invariant with respect to isotopies of $\gamma$.
\end{prp}

\begin{proof}
Let $K_\varepsilon$ denote the compact subset of $C_6(S^3)$ consisting of $6$-tuples for which each pair of points are no less than a distance of $\varepsilon $ apart. From Proposition \ref{bound}, we see that $Q_\gamma\cap M$ lies inside $K_{2\tau(\gamma)}$, and therefore lies strictly within the interior of $K_{\tau(\gamma)}$. This means that we can make $M$ transverse to $Q_\gamma$ by applying a small isotopy that only modifies the inside of $K_{\tau(\gamma)}$. Let $M'$ be the perturbation of $M$ we get after applying such an isotopy. Then $M'\cap Q_\gamma$ is an oriented 1-dimensional manifold that lies inside of $Q_\gamma$, and it therefore defines a homology class which we denote $[M\cap Q_\gamma]\in H_1(Q_\gamma)$.  First, we must check that this homology class does not depend on $M'$. Let $M'_0$ and $M'_1$ be two choices for $M'$. Then we let $M(t)$ be an isotopy so that $M(0) = M'_0$ and $M(1) = M'_1$, and $M(t)$ agrees with $M$ outside of $K_{\tau(\gamma)}$. Then $M(t)$ for $t\in [0,1]$ can be thought of as a submanifold $\mathbf{M}$ of $C_6(S^3) \times [0,1]$, and we can perturbe this manifold slightly to make a submanifold $\mathbf{M}'$ of $C_6(S^3) \times [0,1]$ that intersects $Q_\gamma\times[0,1]$ transversely. The intersection then yields a cobordism from $M_0'\cap Q_\gamma$ to $M_1'\cap Q_\gamma$ over $Q_\gamma$. This demonstrates that the homology class does not depend on the choice of $M'$. Now, we must show that the homology class we have defined is an isotopy invariant of $\gamma$. Take an isotopy $\gamma_t$. As the manifolds $Q_{\gamma_t}$ vary with $t\in [0,1]$ we get a submanifold $\mathbf{Q}$ of $C_6(S^3) \times [0,1]$. Now, let $\varepsilon = \min_{t\in [0,1]} \tau(\gamma_t)$, which exists and is positive because $\tau$ is continuously dependent on $\gamma$. The intersection $\mathbf{Q}\cap (M\times[0,1])$ is inside the interior of the compact set $K_{\varepsilon}\times[0,1]$ so we can let $\mathbf{M}$ be a perturbation of $M\times [0,1]$ which is unperturbed outside of $K_\varepsilon \times[0,1]$ and is transverse to $\mathbf{Q}$. Then, $\mathbf{Q}\cap \mathbf{M}$ is a cobordism in $\mathbf{Q}$ from a manifold in $Q_{\gamma_0}$ representing $[M\cap Q_{\gamma_0}]$ to a manifold in $Q_{\gamma_1}$ representing $[M\cap Q_{\gamma_1}]$. Furthermore, $\mathbf{Q}$ is homotopy equivalent via the inclusions to both $Q_{\gamma_0}$ and to $Q_{\gamma_1}$. This demonstrates that the homology class in $H_1(Q_\gamma)$ is isotopy invariant with respect to $\gamma$, and is therefore a knot invariant. 
\end{proof}

We have a canonical, continuously dependent diffeomorphism $Q_\gamma\to \R^5 \times S^1$ from the proof of Proposition \ref{diffeoQ}. Therefore, the knot invariant $[M\cap Q_\gamma]$ can essentially be thought of as an integer, but with sign dependent on our choice of orientations for $Q_\gamma$, $M$, and $S^1$. 

\section{Identifying the Intersection Class}
The invariant constructed in Proposition $\ref{inv}$ will be denoted by $\mathfrak{K}(\gamma) = [M\cap Q_\gamma]$. In this section, we will identify $\mathfrak{K}(\gamma)$ as the quadratic term of the Conway polynomial multiplied by some fixed element of $H_1(Q_\gamma) \simeq \Z$. 

\begin{prp}
$\mathfrak{K}(\gamma)$ is a rank 2 Vassiliev invariant. 
\end{prp}

\begin{proof}
Fix a stereographic inclusion $\R^3\subseteq S^3$.

Let $\gamma: \R\to \R^3$ be a smooth $\Z$-periodic parameterization with non-vanishing derivative for an immersed circle in $\R^3$ with exactly three transverse self-intersection points $p_1,p_2,p_3$. Now, let $v_1,v_2,v_3$ be nonzero vectors such that both tangent lines to $\gamma$ at $p_i$ are orthogonal to $v_i$ for each $i$. Take small disjoint closed spherical neighborhoods $N_1, N_2, N_3$ centered around the points $p_1,p_2,p_3$ respectively. We can assume that these neighborhoods are small enough that for each ball $N_i$, the curve $\gamma$ intersects the boundary transversely at exactly four points, and the two intervals in $S^1$ between these points map to intersecting strands of the curve $\gamma$, and at no point in $N_i$ is the tangent line of $\gamma$ parallel to $v_i$. Lastly, at each point $p_i$, we make an arbitrary distinction between the two intersecting strands in $N_i$, calling them strand 1 and strand 2. We choose six smooth periodic functions, labeled $b_{i,j}: \R\to \R_{\geq 0}$ where $i$ ranges from 1 to 3, and corresponds to the indices of the self-intersection points, and $j$ ranges from 1 to 2, denoting the strand. We require that these functions have support in the pre-image of their corresponding strand, and are nonzero at the points that map to $p_i$. Now, given an $\varepsilon > 0$, we define $\gamma_{s_1,s_2,s_3}(t) = \gamma(t) + \varepsilon \sum_{i = 1}^3 s_i(b_{i,1}(t) - b_{i,2}(t))v_i$. We will think of $(s_1,s_2,s_3,t)\mapsto \gamma_{s_1,s_2,s_3}(t)$ as a function $[-1,1]^3\times \R\to S^3$. We will fix $\varepsilon$ to be small enough that the following properties hold. 
\begin{itemize}
\item[1)] If $\gamma(t)\in N_i$ then $\gamma_{s_1,s_2,s_3}(t) \in N_i$.
\item[2)] For each $i = 1,2,3$, if $\gamma(t)$ and $\gamma(t')$ are both in $N_i$, and $\gamma_{s_1,s_2,s_3}(t) = \gamma_{s_1',s_2',s_3'}(t)$ and we have $s_is_i' > 0$, then $t = t'\pmod{\Z}$ and $s_i = s_i'$.
\end{itemize}

We will call $\gamma_{s_1,s_2,s_3}$ the resolution of the self intersecting curve $\gamma$. To prove that $\mathfrak{K}$ is a rank 2 invariant, we must prove that the following formula holds. 

\begin{equation}\sum_{s_1,s_2,s_3\in \{-1,1\}} s_1s_2s_3\cdot \mathfrak{K}(\gamma_{s_1,s_2,s_3}) = 0 \label{vass}\end{equation}

 Select a point $p$ on $\gamma$ away from the neighborhoods $N_1,N_2,N_3$. If $Q_{\gamma_{s_1,s_2,s_3}}^p$ denotes the submanifold of $Q_{\gamma_{s_1,s_2,s_3}}$ consisting of $6$-tuples where the first term is $p$, then as a consequence of Proposition \ref{inv}, we see that the oriented intersection number of $Q_{\gamma_{s_1,s_2,s_3}}^p$ and $M$, which we will denote by $[M\cap Q_{\gamma_{s_1,s_2,s_3}}^p]$, equals $\mathfrak{K}(\gamma_{s_1,s_2,s_3})$. To prove our proposition, we will describe a way to perturb the manifolds $Q_{\gamma_{s_1,s_2,s_3}}^p$ so that they overlap in a nice way that makes Equation \ref{vass} obvious.

Let $t_0$ be a real number with $\gamma(t_0) = p$. Now, if we fix a 5-tuple $ (t_1,...,t_5)$ of real numbers, we define $N_i$ to be ``doubled'' if the two intersecting segments of $\gamma$ passing through $N_i$ each contain at least one point from $\{\gamma(t_1),...,\gamma(t_5)\}$. Now, let $\theta_1,\theta_2,\theta_3$ be smooth functions which map $5$-tuples $ (t_1,...,t_5)$ of real numbers into the closed interval $[0,1]$, such that if $N_i$ is not doubled with respect to $(t_1,...,t_5)$, then $\theta_i(t_1,...,t_5) = 0$, and if both points in $S^1$ that map to $p_i$ are represented in $\{t_1,...,t_5\}$, then $\theta_i(t_1,...,t_5) = 1$.

We are now ready to define, for any triple $s_1,s_2,s_3\in \{-1,1\}$, a perturbation of $Q_{\gamma_{s_1,s_2,s_3}}^p$. We will call our perturbation $\tilde{Q}_{\gamma_{s_1,s_2,s_3}}^p$ and we produce it by moving each $6$-tuple of the form $(p,\gamma_{s_1,s_2,s_3}(t_1),...,\gamma_{s_1,s_2,s_3}(t_5))$ to the 6-tuple \begin{equation}(p, \gamma_{s_1\cdot \theta_1(t_1,...,t_5), ..., s_3\cdot \theta_3(t_1,...,t_5)}(t_1), ..., \gamma_{s_1\cdot \theta_1(t_1,...,t_5), ..., s_3\cdot \theta_3(t_1,...,t_5)}(t_5))\label{tuple}\end{equation} via the isotopy $(p, g_1(\alpha,t_1,...,t_5), ..., g_5(\alpha,t_1,...,t_5))$, $\alpha\in [0,1]$ where $$g_i(\alpha,t_1,...,t_5) =  \gamma_{s_1\cdot (1-\alpha(1-\theta_1(t_1,...,t_5))), ..., s_3\cdot (1-\alpha(1-\theta_1(t_1,...,t_5)))}(t_i)$$ For convenience, we will write $\tilde{Q}_{\gamma_{s_1,s_2,s_3}}^p(t_1,...,t_5)$ as shorthand to denote the 6-tuple in (\ref{tuple}).

One way to describe what is geometrically happening here is to think of this isotopy as bringing the submanifold $Q_{\gamma_{s_1,s_2,s_3}}^p$ as close to the subset of 6-tuples that lie on $\gamma$ as we possibly can. Indeed, if there are no doubled neighborhoods, then our perturbed 6-tuple consists only of points that lie in $\gamma$. In order to avoid points coinciding with each other, we have introduced the notion of doubled neighborhoods, and selected our perturbation so that when a doubled neighborhood occurs, the points move back away from $\gamma$ to the resolved arc. We will now describe how this perturbation causes our submanifolds to overlap nicely.

By the reverse pigeonhole principle, every 5-tuple $(t_1,...,t_5)$ has at least one non-doubled neighborhood $N_i$. Suppose $s_1',s_2',s_3'$ are such that we have $s_i  = -s_i'$ and $s_j = s_j'$ for $j\neq i$. Then since $N_i$ is not doubled, we have that  $\tilde{Q}_{\gamma_{s_1,s_2,s_3}}^p(t_1,...,t_5) = \tilde{Q}_{\gamma_{s_1',s_2',s_3'}}^p(t_1,...,t_5)$. This demonstrates that our perturbed manifolds overlap. If $M$ intersects $\tilde{Q}_{\gamma_{s_1,s_2,s_3}}^p$ transversely at the point $\tilde{Q}_{\gamma_{s_1,s_2,s_3}}^p(t_1,...,t_5)$ then it also intersects $\tilde{Q}_{\gamma_{s_1',s_2',s_3'}}$ transversely at the point $\tilde{Q}_{\gamma_{s_1',s_2',s_3'}}^p(t_1,...,t_5)$, and because the parameterization by $t_1,...,t_5$ determines the orientation, the sign of these intersections will be the same. Furthermore, $s_1s_2s_3 = -s_1's_2's_3'$ so the signs for the sum in Equation \ref{vass} cancel. Therefore, all we need to do to prove Equation \ref{vass} is justify the claim that the isotopy $(p, g_1(\alpha,t_1,...,t_5), ..., g_5(\alpha,t_1,...,t_5))$ described above preserves the oriented intersection number of each  perturbed manifold with $M$.  

For convenience, we will write $$g_0(\alpha,t_1,...,t_5) = \gamma_{s_1\cdot (1-\alpha(1-\theta_1(t_1,...,t_5))), ..., s_3\cdot (1-\alpha(1-\theta_1(t_1,...,t_5)))}(t_0) = p$$

To prove that the oriented intersection number is preserved, we can use the same argument as in Propositions \ref{bound} and \ref{inv}. All we need for this argument to work is to show that if two terms of a 6-tuple of the form $(g_0(\alpha,t_1,...,t_5), ..., g_5(\alpha,t_1,...,t_5))$ are sufficiently close together, then the two terms of minimal distance are adjacent with respect to the cyclic ordering. To prove this, observe that the derivative of $ \gamma_{s_1\cdot (1-\alpha(1-\theta_1(t_1,...,t_5))), ..., s_3\cdot (1-\alpha(1-\theta_1(t_1,...,t_5)))}$ is always nonzero, and $g_i(\alpha,t_1,...,t_5) \neq g_j(\alpha,t_1,...,t_5)$ when $t_i \neq t_j \pmod{\Z}$. Now suppose for the sake of contradiction we have a sequence of pairs $(x_1,x'_1), (x_2,x'_2), ...$ each of which are the pair of minimal distance in some 6-tuple $(p, g_1(\alpha,t_1,...,t_5), ..., g_5(\alpha,t_1,...,t_5))$, and each of which are cyclically non adjacent in that 6-tuple, and $\lim_{n\to \infty} |x_n-x_n'| = 0$. Let $r_n,r'_n\in \R/\Z$ and $\alpha_n\in [0,1]$ be such that $$\gamma_{s_1\cdot (1-\alpha_n(1-\theta_1(t_1,...,t_5))), ..., s_3\cdot (1-\alpha_n(1-\theta_1(t_1,...,t_5)))}(r_n) = x_n $$ and $$ \gamma_{s_1\cdot (1-\alpha_n(1-\theta_1(t_1,...,t_5))), ..., s_3\cdot (1-\alpha_n(1-\theta_1(t_1,...,t_5)))}(r'_n) = x'_n $$ Now, if  the distance between $r_n$ and $r'_n$ goes to zero as $n\to\infty$, then the limit points of $x_n$ will be points where the derivative of $\gamma_{s_1\cdot (1-\alpha(1-\theta_1(t_1,...,t_5))), ..., s_3\cdot (1-\alpha(1-\theta_1(t_1,...,t_5)))}$ is zero for some $\alpha$ which yields a contradiction. Otherwise, there will be a subsequence with the distance between $r_{n_i}$ and $r'_{n_i}$ bounded below, and by compactness, we get $g_i(\alpha,t_1,...,t_5)  = g_j(\alpha,t_1,...,t_5)$ with $t_i \neq t_j$ which yields a contradiction.
\end{proof}

\begin{prp}
If $\gamma$ is an unknot, $\mathfrak{K}(\gamma) = 0$.
\end{prp}
\begin{proof}
By Proposition \ref{inv}, we only need to prove that $\mathfrak{K}(\gamma) = 0$ for $\gamma$ being a great circle. For this choice of $\gamma$, the manifold $Q_\gamma$ does not intersect $M$ at all because if three lines intersecting outside of $B^4$ in $\RP^4$ each pass through two points on a great circle in $\partial B^4$, at least one of those lines will have its intersection points adjacent along the great circle, which means the order of the points will be wrong and the 6-tuple will not be in $M$.
\end{proof}

\begin{prp}
If $\gamma$ is a trefoil knot, $\mathfrak{K}(\gamma)$ is a nontrivial element of $H_1(Q_\gamma)$.
\end{prp}
\begin{proof}
By Proposition \ref{inv}, we only need to prove that $\mathfrak{K}(\gamma)$ is a nontrivial element of $H_1(Q_\gamma)$ when $\gamma$ is the embedding of the trefoil in $S^3\subseteq \R^4$ given by $$\gamma(t) = \frac{1}{\sqrt{2}}(\cos(4\pi t),\sin(4\pi t),\cos(6\pi t),\sin(6\pi t))$$ Let $p = \gamma(0)$. Then we want to compute the oriented intersection number of $Q_\gamma^p$ with $M$. First, notice that $Q_\gamma^p\cap M$ contains the 6-tuple $(\gamma(0), \gamma(\frac{1}{6}),  ... , \gamma(\frac{5}{6}))$, and the intersection is transverse at this point. We claim that the remaining points, if they exist, must have an even contribution to $\mathfrak{K}$.  To prove this, we first define two actions of $(\Z/2\Z)$ on $\R^4$. Our actions will be denoted by $\tau_1$ and $\tau_2$, where $\tau_1$ negates the second and fourth coordinates in $\R^4$, and $\tau_2$ negates the third and fourth coordinates in $\R^4$. Now, we extend these two actions to $C_6(S^3)$. The action $\tau_1$ is first applied term-wise to each term in the 6-tuple, and then we apply the permutation $(x_1,...,x_6)\mapsto (x_1,x_6,x_5,x_4,x_3,x_2)$. The action $\tau_2$ is first applied term-wise, and then we apply the permutation $(x_1,...,x_6)\mapsto (x_4,x_5,x_6,x_1,x_2,x_3)$. Combining $\tau_1$ and $\tau_2$, we get a $(\Z/2\Z)^2$ action on both $\R^4$ and $C_6(S^3)$. Inside $C_6(S^3)$, the intersection $Q_\gamma^p\cap M$ is mapped to itself under $\tau_1$.  Let $(Q_\gamma^p \cap M)^{\tau_1}$ denote the subset of $Q_\gamma^p \cap M$ which consists of fixed points under $\tau_1$. We see that any such fixed point has $\gamma(1/2)$ as the 4-th term in the 6-tuple. This lets us conclude that $(Q_\gamma^p \cap M)^{\tau_1}$ is mapped to itself under $\tau_2$. Now, observe that if a point is in $Q_\gamma^p \cap M$ and is fixed by both $\tau_1$ and $\tau_2$, then it is the 6-tuple $(\gamma(0), \gamma(\frac{1}{6}),  ... , \gamma(\frac{5}{6}))$. If we take a small perturbation of $M$ to $M'$ which only modifies $M$ at points where the action of $(\Z/2\Z)^2$ is free and makes the intersection with $Q_\gamma^p$ transverse at these points, then the intersection $Q_\gamma^p\cap M'$ is transverse with finitely many points. These intersection points are acted on by the involution $\tau_1$. The fixed points of $\tau_1$ are acted on by $\tau_2$, which has only one fixed point. Therefore, $|Q_\gamma^p\cap M'| \equiv 1 \pmod{2}$. This proves that $\mathfrak{K}(\gamma)\neq 0$.  
\end{proof}

It is very likely, and intuitive from a geometric perspective, that $(\gamma(0), ..., \gamma(\frac{5}{6}))$ is the only element of $Q_\gamma^p\cap M$. Although counting the points of $Q_\gamma^p\cap M$ is a computationally finite exercise which simply involves solving some trigonometric equations, the equations in question are highly complicated. Without some clever algebraic or geometric trick which simplifies the computation, it is likely that any proof that $\{(\gamma(0), ..., \gamma(\frac{5}{6}))\} = Q_\gamma^p\cap M$ will be computer-assisted. I suspect that such a trick exists, but it has eluded me. 

\begin{prp}
$\mathfrak{K}$ is equal to the quadratic term of the Conway polynomial multiplied by some fixed nontrivial element of $H_1(Q_\gamma)$.
\end{prp}

\begin{proof}
Nontrivial multiples of the quadratic term of the Conway polynomial are characterized by being rank 2 invariants which are trivial on the unknot and nontrivial on the trefoil. \cite{IntroToVass}
\end{proof}

\section{Constructing the Inscribed Trefoil}

In this section, we restrict our attention to a fixed analytic parameterization $\gamma : \R\to \R^3$ of a knot which we can think of as a map $\R\to S^3$ by composing with the stereographic projection.  Furthermore, we will assume the knot parameterized by $\gamma$ has a nontrivial quadratic term for its Conway polynomial, and thus satisfies the criteria for Theorem 1.

We define $\Gamma$ to be the group of M\"obius transformations of $\R^3$. This is the group of transformations (which are defined $S^3\to S^3$ but we think of as transformations on $\R$ defined at all but at most one point) generated by translation, scaling by nonzero real constants, and the operation which consists of mapping to $S^3$ under the stereographic projection, applying an isometric rotation, and then projecting back to $\R^3$. 

We define a spherical trefoil knot to be any $6$-tuple of points in $\R^3$ that all lie on some sphere, and when the points are cyclically connected by line segments, produce a trefoil knot.  

\begin{prp} \label{sphere}
Let $(x_1,...,x_6)$ be a spherical trefoil knot lying on a sphere $P$. Let $\mu$ be a M\"obius transformation that does not take $P$ to a flat plane. Then $(\mu(x_1),...,\mu(x_6))$ is also a spherical trefoil knot. (But the handedness might be different from $(x_1,...,x_6)$.)
\end{prp}

\begin{proof}
First we prove the result for M\"obius transformations that do not turn the sphere inside out. The set of such M\"obius transformations is connected, so we just need to show that as a M\"obius transformation varies and the point of inversion does not pass through $P$, the knot type of the polygonal path generated by a sequence of points on $P$ does not change. If the knot type were to change, there would need to be a time at which two line segments in the polygonal path cross. For this to happen, a 4-tuple of points would need to be coplanar when they were not previously. However, the point of inversion is not on $P$ which implies that our coplanar 4-tuple lies on a circle. However, M\"obius transformations preserve circles so we arrive at a contradiction because the 4-tuple would also have to lie on a circle in $P$ and be coplanar originally. The same argument demonstrates that we now only need to show that there is some M\"obius transformation that turns $P$ inside out and produces a spherical trefoil. The mirror image transformation is such an example.
\end{proof}

We define a closed subset $M_0\subset M$ to be the subset of $M$ consisting of $6$-tuples for which all six points lie on some circle in $S^3$. We also define a ``stereographic trefoil'' to be a $6$-tuple of points in $S^3$ which lie on some 2-sphere and which form a spherical trefoil knot under some choice of stereographic projection to $\R^3$. Due to Proposition \ref{sphere}, any stereographic trefoil yields either a spherical trefoil knot or a coplanar set of points when stereograpically projected to $\R^3$ under an arbitrary choice of stereographic projection. 

\begin{prp}\label{mprime}
For any neighborhood $U$ of the manifold $M$, there is a small isotopic perturbation $M'$ inside $U$ such that $M$ and $M'$ coincide exactly on $M_0$ and every point in $M'\setminus M_0$ is a stereographic trefoil. 
\end{prp}

\begin{figure}[h]
\caption{A trefoil knot is formed by slightly rotating the top of a triangular prism.}
\centering
\includegraphics[scale = 0.75]{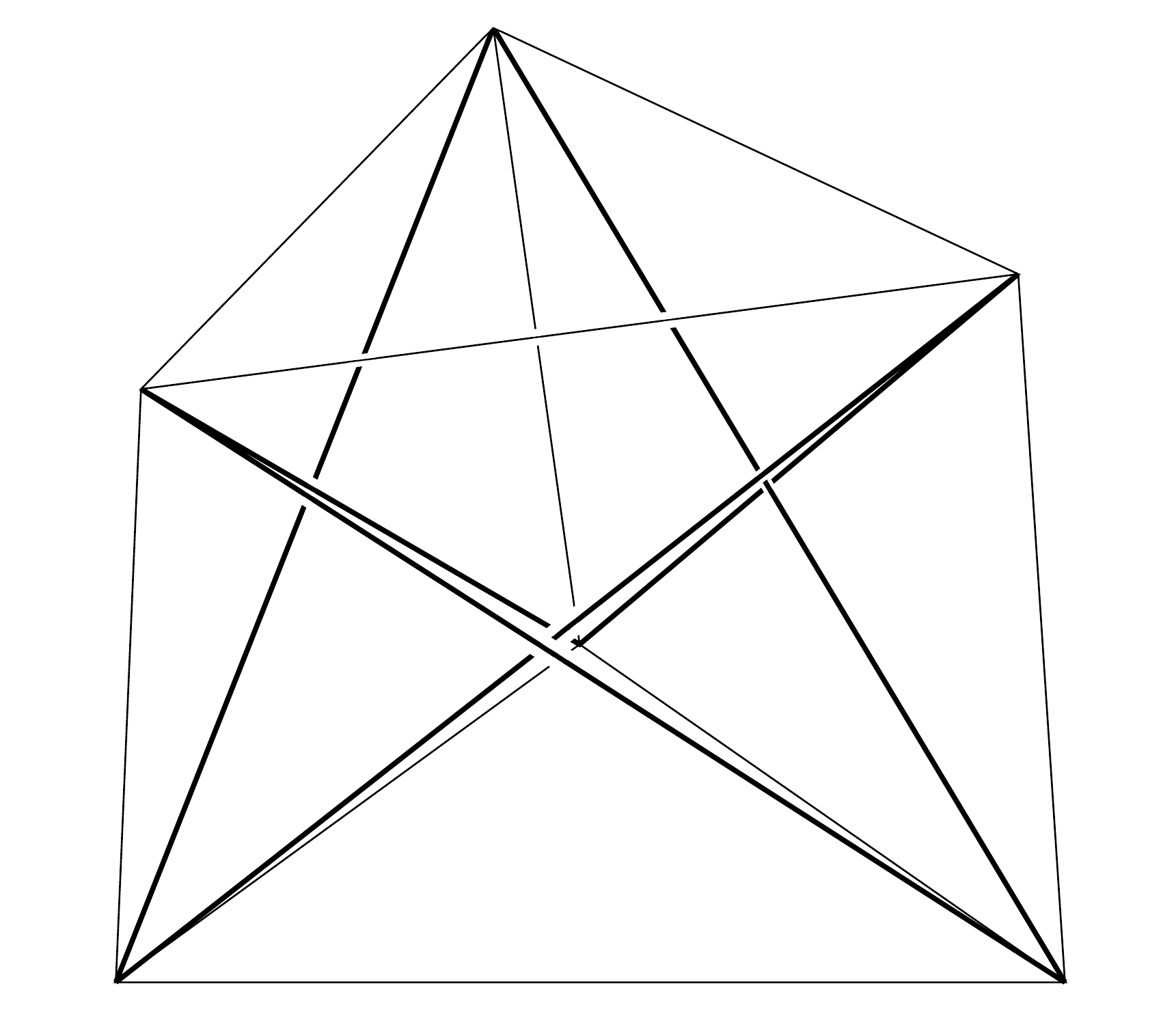}
\end{figure}

\begin{proof}
A point in $M\setminus M_0$ is determined by a point $p$ in $\RP^4\setminus B^4$ and an ordered triple of affinely independent lines $(\ell_1,\ell_2,\ell_3)$ passing through $S^3$ at points $(x_1,...,x_6)$, as well as a specified side of the partition determined by $T_p$ (see the proof of Proposition \ref{diffeoM}). Then the affine 3-space generated by the three lines passes through $S^3$ to form a 2-sphere $P$ containing all points of the 6-tuple.  The 2-spheres $T_p$ and $P$ intersect at a circle $C$ in $P$, and we are given a specified side of $P\setminus C$ induced by the specified side of $T_p$, as well as a specified orientation of $C$. This orientation of $C$ is induced by the intersection of the orientation of $T_p$ which comes from the specified side, and the orientation of the span of the lines induced by their order.  Now, for any $\varepsilon > 0$, we can define $(x_1',...,x_6')$ by isometrically rotating (by angle $\varepsilon$) the specified side of $P\setminus C$ in the direction of the orientation of $C$. For sufficiently small $\varepsilon$, this will yield a stereographic trefoil. See Figure 2. We select a small $\varepsilon$ smoothly dependent on a point in $M$ and equal to zero exactly on $M_0$. Now, applying the transformation $(x_1,...,x_6)\mapsto (x_1',...,x_6')$ at all points in $M\setminus M_0$ yields the desired perturbation $M'$.
\end{proof}

\begin{prp}\label{bigangle}
Suppose the analytic curve with non-vanishing derivative $\gamma$ is tangent to a plane $P$ at the point $\gamma(t_0)$. Let $h(t)$ denote the distance between $\gamma(t)$ and $P$. Let $\theta(t)$ denote the angle between the plane $P$ and the tangent line of $\gamma$ at $\gamma(t)$. Then there exist positive constants $\delta, \varepsilon, C$  so that if we have $|t - t_0| < \delta$ then we have $\theta(t) \geq C\cdot h(t)^{1-\varepsilon}$.
\end{prp}

\begin{proof}
Let $n$ be the minimal natural number so that the $n$-th degree Taylor approximation of $\gamma$ at $t_0$ does not lie in $P$. Then, there exists a constant $C_1 > 0$ so that for $t$ sufficiently close to $t_0$, we have $h(t) \leq C_1 |t|^{n}$. From the derivative, we see that there exists a constant $C_2 > 0$ so that for $t$ sufficiently close to $t_0$, we have $\theta(t) \geq C_2 |t|^{n-1}$. We then have the inequality $(\theta(t)/C_2)^{1/(n-1)} \geq |t| \geq (h(t)/C_1)^{1/n}$ for $t$ sufficiently close to $t_0$, so we see that $\theta(t) \geq C (h(t))^{1 - \frac{1}{n}}$ for some constant $C>0$. Letting $\varepsilon = \frac{1}{n}$, and letting $\delta$ be the required closeness of $t$ and $t_0$, we have the desired inequality.
\end{proof}

\begin{prp}\label{smallangle}
Let $P$ be a plane in $\R^3$, and let $(a,b,c)$ be a triple of affinely independent points in $\R^3$. Let $\theta$ be the angle between $P$ and the plane spanned by $(a,b,c)$. For $x\in \R^3$, let $h(x)$ be the distance between $x$ and $P$. Let $r$ be the radius of the incircle for the triangle with vertices $a,b,c$. Then we have $$ \theta \leq \frac{h(a) + h(b) + h(c)}{r} $$
\end{prp}

\begin{proof}
First, let $P'$ be the plane spanned by $a,b,$ and $c$. If $P$ and $P'$ are parallel, then the proposition is trivial, so assume $P$ and $P'$ intersect at some line $\ell$.  Then, for $x\in P'$ let $h'(x)$ be the distance in $P'$ between $x$ and $\ell$. We have that $h'(x)\sin(\theta) = h(x)$. Furthermore, $0\leq \theta \leq \pi/2$ so $\theta/\sin(\theta) \leq 2$. Therefore, it suffices to prove that $2r\leq h'(a) + h'(b) + h'(c)$. This inequality becomes obvious when one considers the orthogonal projection of $P'$ onto a line orthogonal to $\ell$. In such a projection, the triangle with vertices $a,b,c$ maps to an interval of length less than or equal to $h'(a) + h'(b) + h'(c)$, and this interval contains the projection of the incircle of the triangle which is of length $2r$, therefore, we have the inequality $2r\leq h'(a) + h'(b) + h'(c)$. Multiplying both sides by $\sin(\theta)/r$, and using $\theta \leq 2\sin (\theta)$, we get the desired inequality.  
\end{proof}

We define $Z\subseteq C_3(\R^3)$ to be the set of all triples of affinely independent points on $\gamma$ such that, at each of the points in the triple, $\gamma$ is tangent to the plane passing through the triple. 

\begin{prp}
Every point of $Z$ is isolated. 
\end{prp}

\begin{proof}
By combining Propositions \ref{bigangle} and \ref{smallangle}, we see that if $(\gamma(t_1), \gamma(t_2), \gamma(t_3))\in Z$ spans a plane $P$ and $t_1',t_2',t_3'$ are sufficiently close to $t_1,t_2,t_3$, then if $k\in \{1,2,3\}$ is such that $\gamma(t_k')$ is maximally far from $P$, the angle between the tangent line to $\gamma$ at $t_k'$ and $P$ will be larger than the angle between the plane spanned by $\gamma(t_1'),\gamma(t_2'),\gamma(t_3')$ and $P$. This means that the triple $(\gamma(t_1'),\gamma(t_2'),\gamma(t_3'))$ cannot be in $Z$.
\end{proof}

As a corollary of this proposition, $Z$ is finite because it can be expressed as the set of zeroes of a multivariate analytic function, and such a set is locally connected \cite{analytic} \cite{subanalytic}. However, in the hopes that minimizing the use of analyticity will make it easier to generalize our arguments to larger classes of curves, we will avoid using the finiteness of $Z$. The only fact we will use about about analytically parameterized curves is that they intersect every plane at only finitely many points, and at those points, Proposition \ref{bigangle} is satisfied.  

\begin{prp}\label{rigid}
Let $R$ be the subset of $Q_\gamma$ consisting of 6-tuples which are coplanar in some plane $P$ and which contain a triple of points in $Z$.  Then, there exists a finite sequence of sets $R_1,...,R_n$ with $\bigcup_{i = 1}^n R_i = R$ and such that for each $i$, every point of $R_i$ is an isolated point of the set $R_i$. In particular, we can do this with $n = 20$.
\end{prp}

\begin{proof}
Bijectively associate the numbers $1,...,20$ with the size three subsets of $\{1,...,6\}$. Then, let $R_i$, for $i = 1,...,20$, be the set of coplanar $6$-tuples of points on $\gamma$ for which the triple of points associated to the number $i$ is in $Z$. We see that $\bigcup_{i = 1}^{20}R_i = R$, so we only need to show that every point of $R_i$ is an isolated point of $R_i$.  The curve $\gamma$ passes through every plane only finitely many times, and every point of $Z$ is isolated. This means that every point of $R_i$ must also be isolated. 
\end{proof}

Let $A$ be a subset of a manifold. We write $r(A)$ to denote the Cantor-Bendixson derivative of $A$. \cite{sets} That is to say, the subset of $A$ consisting of all non-isolated points. We write $r^n(A)$ to denote iteration of $r$. 

\begin{prp}\label{removal}
Suppose $X$ and $Y$ are closed subsets of a manifold $M$ so that $X\subseteq Y$ and there exist sets $R_1,...,R_n$ so that $Y\setminus X = \bigcup_{i = 1}^n R_i$ and each point of $R_i$ is an isolated point of the set $R_i$. Then we have $r^n(Y) \subseteq X$. 
\end{prp}

\begin{proof}
We proceed by contradiction. Suppose there exists a point $x \in r^n(Y)\setminus X$. There must exist a number $a$ so that $x\in R_a$. Then, since $X$ is closed and every point of $R_a$ is isolated, there exists an open neighborhood $U$ around $x$ that does not intersect $X$ and has $U\cap R_a = \{x\}$.  We recursively define a sequence of triples $(x_0,U_0, a_0),...,(x_n,U_n, a_n)$.  For the base case, we set $x_0 = x$, $U_0 = U$, and $a_0 = a$.  Now, suppose we have triples $(x_0,U_0,a_0),...,(x_{k-1},U_{k-1},a_{k-1})$ such that $U_{0}\supseteq U_1 \supseteq ... \supseteq U_{k-1}$ are open sets, and for all $i\in \{0,...,k-1\}$ we have $x_i\in r^{n-i}(Y)\setminus X$, and $U_{i}\cap R_{a_i} = \{x_i\}$, and all numbers $a_0,...,a_{k-1}$ are distinct elements of $\{1,...,n\}$. Then, $x_{k-1}\in r^{n+1-k}(Y)\setminus X$ so as long as $k\leq n$ we can find points in $r^{n-k}(Y)\setminus X$ that lie inside any neighborhood of $x_{k-1}$. We can then select $x_k$ to be any point of $r^{n-k}(Y)\setminus X$ that lies inside of $U_{k-1}\setminus\{x_0,...,x_{k-1}\}$. We now select $a_k$ to be any number so that $x_k\in R_{a_k}$. Finally, we select $U_k$ to be any open neighborhood around $x_k$ which lies inside $U_{k-1}\setminus\{x_0,...,x_{k-1}\}$ and is sufficiently small that $U_k\cap R_{a_k} = \{x_k\}$, which can be done because every point of $R_{a_k}$ is isolated. We see that the triples $(x_0,U_0,a_0),...,(x_{k},U_{k},a_{k})$ are such that  $U_{0}\supseteq U_1 \supseteq ... \supseteq U_{k-1}$, and for all $i\in \{0,...,k\}$ we have $x_i\in r^{n-i}(Y)\setminus X$, and $U_{i}\cap R_{a_i} = \{x_i\}$. We also need to show that all numbers $a_0,...,a_{k}$ are distinct. For all $i \in \{1,...,k-1\}$, we observe that $x_k\in U_i$ and $x_k\neq x_i$, and since $R_{a_i}\cap U_i = \{x_i\}\not\ni x_k$ and $x_k\in R_{a_k}$ we have $a_k\neq a_i$. We already had $a_0,...,a_{k-1}$ distinct so we now have $a_0,...,a_k$ distinct. We can continue this induction until we have $(x_0,U_0,a_0),...,(x_{n},U_{n},a_{n})$ with the properties described above. In particular, we have constructed $n+1$ distinct numbers $a_0,...,a_n$, all of which lie in the set $\{1,...,n\}$. This is a contradiction. 
\end{proof}

Let $P$ be an oriented plane in $\R^3$ and let $x = \gamma(t)$ be a point at which $\gamma$ intersects $P$. We call $x$ a two-sided point if for every interval $I$ around $t$, the set $\gamma(I)$ containins points on both sides of $P$. We call $x$ a positive (resp. negative) one-sided point if there is an interval $I$ around $t$ so that $\gamma(I)$ contains points on the top (resp. bottom) side of $P$ but no points on the bottom (resp. top) side of $P$. Since $\gamma$ is analytic, every intersection point falls into one of these three classes because $\gamma(I)$ can never lie entirely inside $P$.

\begin{prp}\label{cases}
For some choice of $M'$ as in Proposition \ref{mprime}, there exists real numbers $ t_1 < ... < t_6 < t_1+1$ so that $(\gamma(t_1),...,\gamma(t_6))\in Q_\gamma \cap M'$ and such that one of the three following possibilities holds.
\begin{itemize}
\item[1)] The 6-tuple $(\gamma(t_1),...,\gamma(t_6))$ forms a trefoil knot.
\item[2)] The 6-tuple $(\gamma(t_1),...,\gamma(t_6))$ lies in some plane $P$, is a stereographic trefoil, and at most three of the points in the 6-tuple are one-sided intersection points in $P$.
\item[3)] The 6-tuple $(\gamma(t_1),...,\gamma(t_6))$ lies in some circle $C$ of finite radius, and if $P$ is the plane containing the circle $C$, at most two of the points in the 6-tuple are one-sided in $P$.
\end{itemize}
\end{prp}

\begin{proof}
By making $M'$ sufficiently close to $M$, we can guarantee that for arbitrarily small perturbations of $M'$, the intersection with $Q_\gamma$ is a 1-manifold with orientation class mapping to $\mathfrak{K}(\gamma)\in H_1(Q_\gamma)$. This means that, although we cannot obviously guarantee that $M'$ and $Q_\gamma$ intersect transversely, we can guarantee that $Q_\gamma\cap M'$ is compact and any neighborhood of $Q_\gamma\cap M'$ contains a loop which is essential in in the space $Q_\gamma$. That is to say, the induced map from cohomology to \v{C}ech cohomology  $ \Z \simeq H^1(Q_\gamma) \to \widecheck{H}^1(Q_\gamma \cap M')$ is nontrivial. We now proceed by contradiction to prove the proposition, so suppose no point satisfying 1, 2, or 3 exists. 

Let $L$ be the subset of $Q_\gamma \cap M'$ consisting of collinear 6-tuples. We claim that for some $n$, we have $$r^n(Q_\gamma \cap M') \subseteq L$$ First, observe that every point in $Q_\gamma\cap M'$ must be coplanar in $\R^3$ because there cannot be any trefoils on $\gamma$. Second, we see that in a stereographic trefoil, there are never four points that lie on a circle because a 6-tuple trefoil in $\R^3$ can never have four coplanar points. This means that any 6-tuple in $Q_\gamma\cap M'\setminus M_0$ lies in a plane $P$, has at least four one-sided points, and these points do not lie on a line. Furthermore, any 6-tuple in $Q_\gamma \cap M_0$ is either colinear, or all points lie on a circle of finite radius in a plane $P$ with at least three one-sided points. Any three points on a circle of finite radius fail to be colinear, so we have that every point of $Q_\gamma \cap M'$ is either a colinear 6-tuple or a coplanar 6-tuple lying on a plane $P$ with a triple of affinely independent one-sided intersection points. Now, from Propositions \ref{rigid} and \ref{removal}, we see that $$ r^n(Q_\gamma \cap M') \subseteq Q_\gamma \cap M' \cap L \subseteq L$$ for some $n$, which demonstrates that our claim is true.

Now, $r$ only removes isolated points and thus cannot affect \v{C}ech cohomology in degrees other than zero. Therefore, we may select a connected set  $X\subseteq r^n(Q_\gamma \cap M')$ such that $H^1(Q_\gamma) \to \widecheck{H}^1(X)$ is nontrivial. From what we have shown so far, $X$ consists only of colinear 6-tuples. Furthermore, $X$ is connected so we may conclude that the linear order (modulo reversal) on the terms of the 6-tuple is constant. This means that there are fixed numbers $e_1,e_2$ from 1 to 6 so that for every 6-tuple $(x_1,...,x_6)\in X$, the points $x_{e_1}$ and $x_{e_2}$ are the endpoints of the minimal line segment containing the 6-tuple. Now, select some $x$ in the curve $\gamma$ which is an extremal point for the convex hull of the knot, and select some $m$ from 1 to 6 such that $m\neq e_1$ and $m\neq e_2$. Since $H^1(Q_\gamma) \to \widecheck{H}^1(X)$ is nontrivial, there is a 6-tuple $(x_1,...,x_6)\in X$ for which $x_m = x$, however no line segment in the convex hull of the knot has $x$ in its interior. This gives us our contradiction.

\end{proof}

\begin{prp}\label{case1}
If there exist real numbers $t_1 < ... < t_6 < t_1+1$ so that the 6-tuple $(\gamma(t_1),...,\gamma(t_6))$ lies in some plane $P$, is a stereographic trefoil, and at most three of the points are one-sided intersection points in $P$, then there exist real numbers $t_1' < ... < t_6' < t_1'+1$ so that $(\gamma(t_1'),...,\gamma(t_6'))$ gives a trefoil knot.
\end{prp}

\begin{proof}
First, notice that the set of trefoils is open. This implies that the set of stereographic trefoils is an open subset of the set of cospherical 6-tuples of points. This tells us that if there exist non-coplanar yet cospherical 6-tuples $(\gamma(t_1'),...,\gamma(t_6'))$ arbitrarily close to  $(\gamma(t_1),...,\gamma(t_6))$, then the proposition will be true. 

To construct such arbitrarily close 6-tuples, choose an orientation for $P$ and draw a circle $C$ in $P$ that separates the positive one-sided points from the negative one-sided points. Such a circle exists because there are only at most three one-sided points. Now, select spheres of arbitrarily large radius that pass through $C$ in such a way that the sphere is above $P$ in the part with the positive one-sided points and below $P$ in the part with the negative one sided points. If we look at the intersection of these spheres with $\gamma$, for sufficiently large radii, there will be choices of intersection points arbitrarily close to $(\gamma(t_1),...,\gamma(t_6))$ because our sphere agrees with the sidedness of all of the points. This gives us the desired $(\gamma(t_1'),...,\gamma(t_6'))$.
\end{proof}

\begin{prp}\label{case2}
If we have $t_1 < ... < t_6 < t_1 + 1$ so that $(\gamma(t_1),...,\gamma(t_6))\in Q_\gamma \cap M_0$ and $(\gamma(t_1),...,\gamma(t_6))$ lies on a circle $C$ of finite radius in such a way that if $P$ is the plane containing $C$ then $P$ has at most two one-sided points from the 6-tuple, then there exist $t_1' < ... < t_6' < t_1'+1$ so that $(\gamma(t_1'),...,\gamma(t_6'))$ gives a trefoil knot.
\end{prp}

\begin{figure}[h]
\caption{A element of $M_0$ on a circle is perturbed with vectors orthogonal to the plane of the circle to produce a trefoil.}
\centering
\includegraphics[scale = 1]{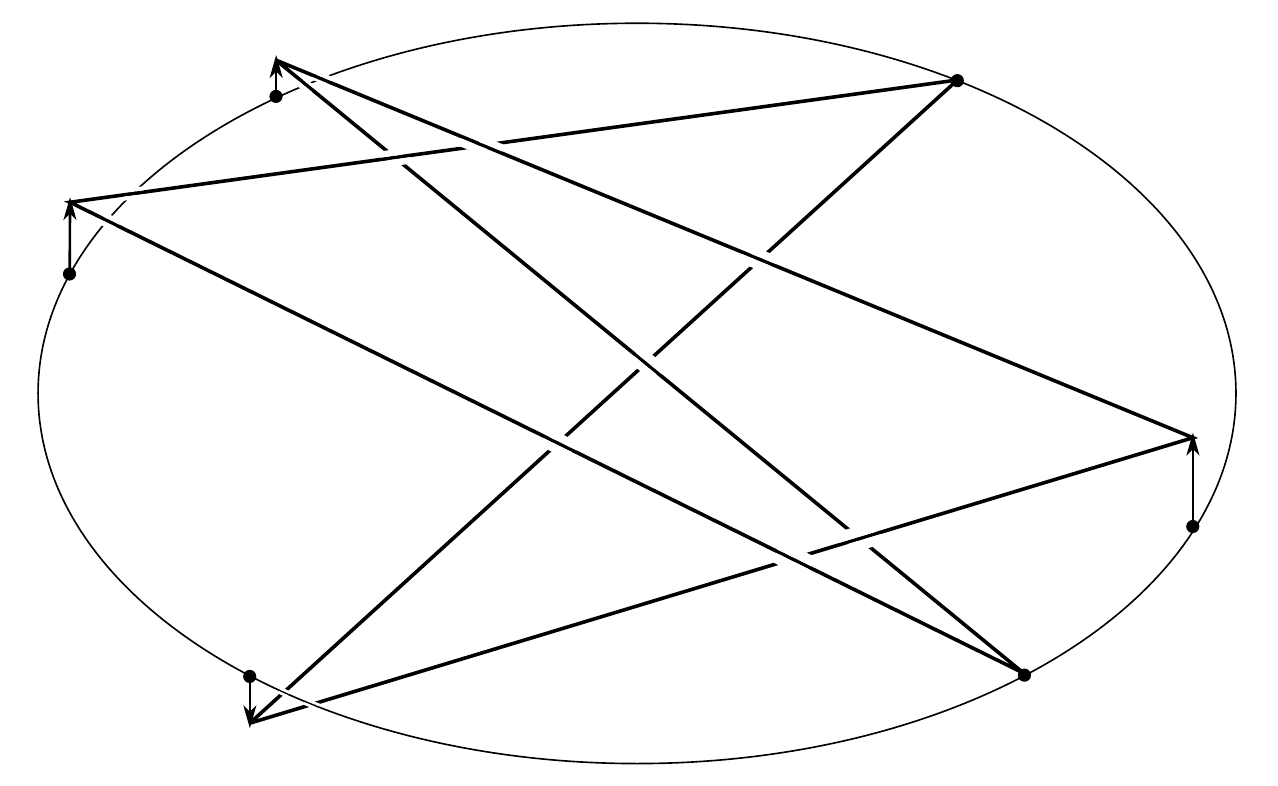}
\end{figure}

\begin{proof}
First, observe that there exists a $6$-tuple of vectors $(v_1,...,v_6)$ orthogonal to $P$ so that for every $\varepsilon > 0$ the 6-tuple $(x_1(\varepsilon),...,x_6(\varepsilon)) = (\gamma(t_1) + \varepsilon v_1,...,\gamma(t_6) + \varepsilon v_6)$ gives a trefoil. See Figure 3 for an example. Such 6-tuples of vectors are preserved under translation perpandicular to $P$ and reflection across a vector subspace paralell to $P$. Therefore, we can assume without loss of generality that $v_i$ points towards the same side as the sidedness of $\gamma(t_i) $ whenever $\gamma(t_i) $ is a one-sided point. (We can translate and reflect to get this condition on both one-sided points.) Now, we create $(\gamma(t_1'(\varepsilon)),...,\gamma(t_6'(\varepsilon)))$ by selecting $\gamma(t_i'(\varepsilon))$ to be the closest point on $\gamma$ to $x_i(\varepsilon)$ that lies in the plane paralell to $P$ that contains $x_i(\varepsilon)$.  By our sidedness assumption, this will always be possible for sufficiently small $\varepsilon$. Now, we claim that for sufficiently small $\varepsilon$, the 6-tuple $(\gamma(t_1'(\varepsilon)),...,\gamma(t_6'(\varepsilon)))$ is a trefoil. 

To see that this is true, consider an affine transformation that fixes $P$ and stretches the component orthogonal to $P$. We may select such a transformation $T_\varepsilon$ depending on $\varepsilon$ so that the 6-tuple $(T_\varepsilon x_1(\varepsilon), ... ,T_\varepsilon x_6(\varepsilon))$ does not depend on $\varepsilon$. Furthermore, since $T_\varepsilon$ does not modify the component of vectors paralell to $P$, we have that $(T_\varepsilon \gamma(t_1'(\varepsilon)),...,T_\varepsilon \gamma(t_6'(\varepsilon)))$ limits to the same trefoil as $(T_\varepsilon x_1(\varepsilon), ... ,T_\varepsilon x_6(\varepsilon))$. Since the set of trefoils is open, this implies that $(T_\varepsilon \gamma(t_1'(\varepsilon)),...,T_\varepsilon \gamma(t_6'(\varepsilon)))$, and thus $(\gamma(t_1'(\varepsilon)),...,\gamma(t_6'(\varepsilon)))$, is a trefoil for small $\varepsilon$.
\end{proof}

We now have a proof of Theorem 1. 

\begin{proof}[Proof of Theorem 1]
Combine propositions \ref{cases}, \ref{case1}, and \ref{case2}.
\end{proof}

\section{Questions}

The theorem we have presented in this paper bears a strong resemblance to a theorem from  \cite{linking} in which it is shown that if a knot has a nontrivial quadratic term of its Conway Polynomial, then it has an alternating quadrisecant. Indeed, alternating quadrisecants are very close to being inscribed trefoil knots. To see why this is the case, if one doubles up the middle two points in an alternating quadrisecant, then there will be a small perturbation of this configuration of points in $\R^3$ that creates a very skinny trefoil knot. It seems quite likely that there is an alternate proof of Theorem 1 that comes from a careful analysis of a knot in the neighborhood of an alternating quadrisecant. We therefore make the following conjecture which would connect these two results in the most obvious way possible.

\begin{cnj}
If $\gamma: \R\to \R^3$ is an analytic $\Z$-periodic parameterization with non vanishing derivative for a knot with nontrivial quadratic term of its Conway polynomial, then for every $\varepsilon > 0$, there exist two sequences of real numbers $t_1 < t_2 < ... < t_6 < t_1 + 1$ and $s_1 < ... < s_4 < s_1 + 1$ so that the points $\gamma(t_1) , ... , \gamma(t_6)$ form a trefoil knot when connected cyclically by line segments, the points $\gamma(s_1),\gamma(s_3),\gamma(s_2),\gamma(s_4)$ lie in order on some line, and the distances $\min_{j = 1,2,3,4} |\gamma(t_i) - \gamma(s_j)|$ for $i = 1,...,6$ are all less than $\varepsilon$. 
\end{cnj}

In \cite{alternating}, Elizabeth Denne showed that all nontrivial knots have alternating quadrisecants. Therefore, if there is indeed a connection between alternating quadrisecants and inscribed trefoils, it may be possible to drop the assumption of nontriviality of the quadratic term of the conway polynomial from Theorem 1, replacing it with the nontriviality of the knot.  

It is not easy to generalize these results to arbitrary smooth parameterizations, let alone continuous parameterizations. We therefore ask, all other conditions from Theorem 1 unchanged, what degree of regularity is required from $\gamma$ to obtain the result?  

A corollary of our main theorem is that every analytically parameterized trefoil knot has an inscribed trefoil knot. However, our proof gives no control over the handedness of the inscribed trefoil. We therefore make the following conjecture. 

\begin{cnj}
Let $\gamma: \R\to \R^3$ be an analytic $\Z$-periodic function with non-vanishing derivative which parameterizes a right-handed trefoil. Then there exists a sequence of numbers $0\leq t_1 < t_2 < ... < t_6 < 1$ so that the polygonal path obtained by cyclically connecting the points $\gamma(t_1), \gamma(t_2), ..., \gamma(t_6)$ by line segments is also a right-handed trefoil. 
\end{cnj}

Finally, are there nontrivial knot types $K_1$ and $K_2$, and some $n\geq 7$, such that every analytic parameterization of a knot of type $K_1$ has $n$ points on it that when connected cyclically in order by line segments yield a knot of type $K_2$? 

\nocite{*}

\bibliography{Refrences}{}
\bibliographystyle{plain}

\end{document}